\documentclass[12pt]{amsart}

\usepackage[T1]{fontenc}
\usepackage{lmodern}

\usepackage{amsmath} 
\usepackage{amsthm} 
\usepackage{amssymb} 

\usepackage{microtype} 
\usepackage{pinlabel} 

\usepackage{MnSymbol} 

\usepackage[scaled=0.9]{sourcecodepro} 

\usepackage[hidelinks, pagebackref]{hyperref}

\makeatletter
\define@key{href}{font}{#1}
\makeatother
\usepackage{xpatch}
\newcommand\hrefdefaultfont{\ttfamily}
\xpatchcmd\href{\setkeys{href}{#1}}{\setkeys{href}{font=\hrefdefaultfont,#1}}{}{\fail}

\renewcommand*{\backref}[1]{}
\renewcommand*{\backrefalt}[4]{
  \ifcase #1 
  [No citations.]
  \or [#2]
  \else [#2]
  \fi }

\let\originalleft\left
\let\originalright\right
\renewcommand{\left}{\mathopen{}\mathclose\bgroup\originalleft}
\renewcommand{\right}{\aftergroup\egroup\originalright}

\newcommand{\calB}{\mathcal{B}}

\newcommand{\QQ}{\mathbb{Q}}
\newcommand{\RR}{\mathbb{R}}

\newcommand{\ZZ}{\mathbb{Z}}

\newcommand{\from}{\colon} 
 
\newcommand{\bdy}{\partial} 

\newcommand{\SL}{\operatorname{SL}} 

\input{header_article.tex}

\theoremstyle{plain}
\newtheorem{XXXtheoremQED}[equation]{Theorem} 
  {\pushQED{\qed}\begin{XXXtheoremQED}}
  {\popQED\end{XXXtheoremQED}}

\newcommand{\fakeenv}{} 

{ 
 \renewcommand{\fakeenv}{#2} 
 \theoremstyle{plain} 
 \newtheorem*{\fakeenv}{#1~\ref{#2}} 
 \begin{\fakeenv}
}
{
 \end{\fakeenv}
}

\newenvironment{restated}[2]  
{ 
 \renewcommand{\fakeenv}{#2} 
 \theoremstyle{definition} 
 \newtheorem*{\fakeenv}{#1~\ref{#2}} 
 \begin{\fakeenv}
}
{
 \end{\fakeenv}
}

\usepackage{enumitem}
\usepackage{color}
\usepackage{subfig, caption}
\usepackage{wrapfig}
\usepackage{pdflscape}
\usepackage{array}
\usepackage{MnSymbol}
\usepackage{stmaryrd}
\usepackage{graphicx}
\usepackage{stackengine}
\usepackage{tikz-cd}
\usepackage{multirow}
\usepackage[new]{old-arrows}

\newcommand{\aut}{\mathsf{Aut}}  
\newcommand{\ch}{\mathsf{ch}}  
\newcommand{\rk}{\mathsf{rk}\,}   
\newcommand{\br}{\mathrm{br}\,}   
\newcommand{\Homv}{v} 
\newcommand{\Homw}{w} 

\stackMath

\newcommand{\bigonx}[1]{\stackon[1pt]{#1}{\kern0.32em\bigon}}

\title{Non-standard bi-orders on punctured torus bundles}
\author{Jonathan Johnson and Henry Segerman} 
\date{\today}

\begin{document}

\begin{abstract}
Results of Perron and Rolfsen imply that untwisted hyperbolic once-punctured torus bundles over the circle have bi-orderable fundamental groups.
They do this by showing that the action of the monodromy preserves a “standard” bi-ordering formed using the lower central series of the free group.
Here we investigate other bi-orderings that punctured torus bundle groups can have.
We show that for every such bi-ordering, the largest and second largest proper convex subgroups match the corresponding convex subgroups for a standard bi-ordering.
Moreover, if there exists a third largest convex subgroup, it must also match the third largest convex subgroup for a standard bi-ordering.
However, we also show that these groups admit non-standard bi-orderings.
\end{abstract}

\maketitle

\section{Introduction}

Much progress has been made on determining when a three-manifold group is \emph{left-orderable}, that is, having a strict total order invariant under left (but not necessarily right) multiplication.
In fact, there is a conjectured answer to when a three-manifold group is left-orderable \cite{BGW13}.
However, comparatively little is known about the symmetrized version of this property.

\begin{definition}
	\label{Def:BiOrderable}
    A group is \emph{bi-orderable} if there is a strict total order of its elements invariant under left and right multiplication.
\end{definition}

Most work on bi-orderability to date has focused on the question of which three-manifolds have bi-orderable fundamental groups.
Boyer, Rolfsen, and Wiest~\cite{BRW05} classified all Seifert fibered spaces with bi-orderable fundamental groups.
Perron and Rolfsen~\cite{PerRolf03} proved a sufficient condition for bi-orderability of the groups of fiber bundles over the circle.
Later, Clay and Rolfsen~\cite{ClaRol12} proved an obstruction for bi-orderability of these fiber bundle groups.
Linnell, Rhemtulla, and Rolfsen~\cite{LRR08} generalized Perron and Rolfsen's result, in particular to address some non-fibered manifolds.
These results have been used to show that many link complements have bi-orderable groups \cite{CDN16, John23a, John23b, KinRolf18, Yam17}.

In this paper we turn from asking which three-manifolds have bi-orderable fundamental groups to the question of classifying the bi-orders on a bi-orderable fundamental group. In particular, we look at punctured torus bundles. 
Clay, Perron, and Rolfsen's results~\cite{ClaRol12, PerRolf03} together determine that the punctured torus bundles with bi-orderable fundamental groups are the untwisted hyperbolic punctured torus bundles (\refprop{BiOrderablePuncturedTorusBundles}).
Thus for the rest of the paper we will assume that our punctured torus bundles are of this form.

The following notions to do with \emph{convexity} are key to describing the structure of bi-ordered groups.

\begin{definition}
	\label{Def:Convex}
    A subgroup $C$ of a bi-ordered group $(G,<)$ is \emph{convex with respect to $<$} if for every $a,b\in C$ and $g\in G$, when $a<g<b$ then $g\in C$.
    We call a convex subgroup $C$ \emph{maximal} if $C$ is a proper subgroup and the only convex subgroups containing $C$ are $C$ and $G$.
\end{definition}

Our first result tells us that there is no choice in the maximal convex subgroup for a bi-orderable hyperbolic punctured torus group. We can then take that maximal convex subgroup and ask what \emph{its} maximal convex subgroup is. Again there is no choice.

\begin{theorem} \label{Thm:toporder}
	Let $M$ be an untwisted hyperbolic punctured torus bundle 
	with fundamental group
	\[
		\pi_1(M)\cong G\rtimes \ZZ 
	\]
	where $G$ is the fundamental group of the punctured torus.
	Let $<$ be a bi-ordering of $\pi_1(M)$. 
	\begin{enumerate}
	    \item The maximal convex subgroup of $\pi_1(M)$ with respect to $<$ is $G$.

		\item Let $<_G$ be the induced bi-order of $<$ on $G$.
	    The maximal convex subgroup of $G$ with respect to $<_G$ is $G_2=[G,G]$.
    \end{enumerate}
\end{theorem}

A maximal convex subgroup is normal (\refprop{FiniteGeneratorMaxConvex}), and it follows from a theorem of H\"older that the quotient is abelian~\cite{Holder1901}. 
Thus the ``most significant digit'' of any such bi-order is given by a bi-ordering of this finitely generated free abelian group; such bi-orders are well understood.

One might ask if the pattern of maximal convex subgroups being determined continues. 
Our next result tells us that if $G_2$ has a maximal convex subgroup then, again, there is no choice.

\begin{theorem} \label{Thm:G3Convex}
	Let $M$ be an untwisted hyperbolic punctured torus bundle 
	with fundamental group
	\[
		\pi_1(M)\cong G\rtimes_h \ZZ 	
	\]
	where $G$ is the fundamental group of the punctured torus.
	Let $<$ be a bi-ordering of $\pi_1(M)$.
	Let $<_2$ be the induced bi-order of $<$ on $G_2=[G,G]$.
	Suppose that $C$ is a maximal convex subgroup of $G_2$ with respect to $<_2$.
	Then $C$ is $G_3=[G_2,G]$.
\end{theorem} 

In \refthm{toporder} the groups $\pi_1(M)$ and $G$ must have maximal convex subgroups because they are finitely generated (\refprop{FiniteGeneratorMaxConvex}).
However, $G_2$ is not finitely generated, leaving open the possibility that it fails to have a maximal convex subgroup.

In \refsec{nonstandard}
we build bi-orders on $\pi_1(M)$ for which $G_2$ does not have a maximal convex subgroup (\refcor{NoMaximalConvex}).
Our examples are not obtainable by existing construction methods in the literature.
Perron and Rolfsen~\cite{PerRolf03} show that untwisted hyperbolic punctured torus bundles are bi-orderable by showing that the monodromy of each of these bundles preserves at least one bi-ordering of the free group.
These bi-orderings are \emph{standard}, meaning that every term of the lower central series is convex (although each term may not be maximal in the previous term); see \refdef{Standard} and \refprop{StandardConvexDef} for the details.
In a natural extension of this terminology, we say that the fundamental group of a punctured torus bundle is \emph{standard} if $h_*$ preserves a standard bi-ordering on the free group $G$.

\begin{theorem} \label{Thm:nonstandard}
Let $M$ be an untwisted hyperbolic punctured torus bundle.
Then $\pi_1(M)$ admits a non-standard bi-ordering.
\end{theorem}

\subsection{Bi-orders and eigenvalues}

Here we give some relevant context to our results and discuss possible future directions.

In this section, let $M$ be a fiber bundle over the circle whose fiber is a once punctured surface $S$ (not necessarily a genus one surface), and  let $h$ be the monodromy of this bundle.
As in the genus one case, $\pi_1(M)$ factors into a semidirect product $\pi_1(M)\cong G\rtimes_{h_*}\ZZ$ where $G\cong \pi_1(S)$ is a free group of rank twice the genus of $S$.
Thus there is a natural map $\pi_1(M)\to\ZZ$ obtained by projecting $G\rtimes_h \ZZ$ onto the $\ZZ$ factor.
Therefore a bundle over the circle has a canonical Alexander polynomial $\Delta_M(t)$.
Perron and Rolfsen, and Clay and Rolfsen prove the following.

\begin{theorem}[\cite{PerRolf03}, Theorem 1.1] 
\label{Thm:PerRolf}
	If all the roots of $\Delta_M(t)$ are real and positive then $\pi_1(M)$ admits a standard bi-ordering.
\end{theorem}

\begin{theorem}[\cite{ClaRol12}, Theorem 1.1]
\label{Thm:ClayRolf}
	If $\pi_1(M)$ is bi-orderable then $\Delta_M(t)$ has at least one real and positive root.
\end{theorem}

\begin{remark}
	These results are stated for fibered knot complements but the proofs can be applied to arbitrary fiber bundles over the circle without any alterations.
\end{remark}

\refthm{PerRolf} is a sufficient but not necessary condition for $\pi_1(M)$ to admit a standard bi-ordering.
However, to know whether or not $\pi_1(M)$ admits a standard bi-ordering we only need to investigate the induced map $h_+\from H_1(S;\ZZ)\to H_1(S;\ZZ)$. (See \refprop{PreserveStandardOrders} for more details.)

\begin{question} \label{Que:BOHplus}
	Suppose that $S$ is a once-punctured surface, and
	suppose that $M$ and $M'$ are $S$ bundles over the circle with monodromies $h$ and $h'$ respectively. 
	Is it possible that $\pi_1(M')$ is bi-orderable, 
	$\pi_1(M)$ is not bi-orderable, but in homology we have that $h_+=h_+'$? 
\end{question}

\begin{remark}
If we allow $S$ to have multiple punctures in \refque{BOHplus} then the answer is ``yes''.
In particular, let $\br(\beta)$ be the link in $S^3$ formed by the closure of a braid $\beta$ along with the braid axis.
Let $M$ and $M'$ be the exteriors of $\br(\sigma_1\sigma_2\sigma_1)$ and $\br(\sigma_1\sigma_2\sigma_1^{-1})$ respectively.
The manifolds $M$ and $M'$, pictured in \reffig{TwoBundles}, are $S$ bundles over the circle where $S$ is a 4-punctured sphere.
For both bundles, the induced maps $h_+, h_+'\from H_1(S;\ZZ)\to H_1(S;\ZZ)$ are described by the following matrix:
\[
	h_+=h_+'=\left(\begin{array}{ccc} 0 & 0 & 1 \\ 0 & 1 & 0 \\ 1 & 0 & 0 \end{array}\right)
\]
Kin and Rolfsen showed that $\pi_1(M)$ is bi-orderable \cite[Theorem 4.10]{KinRolf18} but $\pi_1(M')$ is not bi-orderable \cite[Proposition 4.4, Corollary 4.15]{KinRolf18}.
\end{remark}

\begin{figure}[t]
\subfloat[$M$]
{
\includegraphics[scale=1]{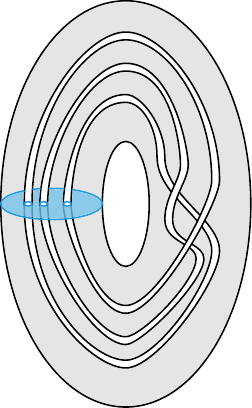}
}
\quad
\subfloat[$M'$]{
\includegraphics[scale=1]{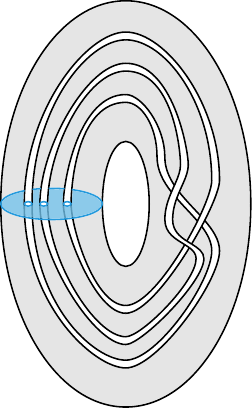}
}
\caption{Bundles over the circle}
\label{Fig:TwoBundles}
\end{figure}

By \refprop{PreserveStandardOrders}, an example answering Question \ref{Que:BOHplus} in the affirmative would have to involve a non-standard bi-ordering.
\refthm{nonstandard} confirms that non-standard bi-orderings exist, although in our examples $M$ is hyperbolic and $S$ is of genus one, a setting in which the answer to Question \ref{Que:BOHplus} is known to be ``no''. (See \refprop{BiOrderablePuncturedTorusBundles}: $h_+$ determines whether or not $M$ is untwisted.)

A negative answer to Question \ref{Que:BOHplus} would have to arise from restrictions that the map on homology $h_+$ places on the non-standard bi-orderings of $\pi_1(M)$. In our examples, $h_+$ puts severe restrictions on what is possible. One might expect the same to be true in higher genus.

\begin{question}
	What other non-standard bi-orderings of a hyperbolic punctured torus (or other surface) bundle group are there beyond those generated by our construction?
\end{question}

\subsection{Structure of the paper}

We begin with an exposition of the interplay between bi-orderings of groups and their convex subgroups in \refsec{BOsandConvex}.
In \refsec{PuncTorusBundles} we recall the definition and some results about punctured torus bundles.
\refsec{G2MaxConvex} focuses on the restrictions that the monodromy imposes on the types of bi-orderings which can be put on the rank 2 free group.
Theorems \ref{Thm:toporder} and \ref{Thm:G3Convex} are proved in \refsec{ProofOfTopOrder} and \refsec{ProofOfG3Convex} respectively.
We describe non-standard bi-ordering of a hyperbolic punctured torus bundle group in \refsec{nonstandard} and so prove \refthm{nonstandard}.

\subsection{Acknowledgments}

The authors thank Danny Calegari and Neil Hoffman for helpful conversations.
The authors were supported in part by National Science Foundation grants DMS-2213213 and DMS-2203993 respectively.

\section{Bi-Orders and Convex Subgroups} \label{Sec:BOsandConvex}

In this section, we state some preliminary facts about bi-orderings and convex subgroups of bi-ordered groups.
In particular, we explore the interplay between the bi-orderings of groups $G$ and $H$ related by a surjective homomorphism $f\from G \to H$.
For a general reference for many of these results see Section 2.2 of Clay and Rolfsen's book~\cite{ClayRolf16}.
We start by stating an alternative way to describe a bi-ordering of a group.

\subsection{Positive Cones and Induced Orderings}

While the most natural method to define a bi-ordering of a group is to specify a total ordering, a bi-ordering is also specified by its set of positive elements.

\begin{definition}
\label{Def:PositiveCone}
    A subset $P\subset G$ is a \emph{(conjugate invariant) positive cone} if it satisfies the following:
    \begin{enumerate}
        \item \label{Itm:Closed}
        $P\cdot P \subset P$
        \item \label{Itm:Partition}
        $G = P \sqcup P^{-1} \sqcup \{1\}$
        \item \label{Itm:Invariant}
        $g^{-1}Pg = P$ for all $g$ in $G$. \qedhere
    \end{enumerate}
\end{definition}

A bi-order $<$ of $G$ corresponds to a unique positive cone by $P_< = \{ g \in G \mid 1 < g\}$.
In the other direction, a bi-order can be obtained from a (conjugate invariant) positive cone $P$ by defining $a <_P b$ if and only if $a^{-1} b \in P$.
We use $(G,<)$, $(G,P)$, or $(G,<,P)$ to mean a group with bi-order $<$ and/or its corresponding positive cone $P$.

\begin{remark}
	The term ``positive cone'' is often used to define the positive elements of a left-ordering of a group.
	In this case, condition (\ref{Itm:Invariant}) is not required.
	The conjugate invariant positive cones of a group are precisely the positive cones associated to left-orders which also happen to be bi-orders.
	Since we only deal with positive cones of bi-orders, all positive cones in this paper are assumed to be conjugate invariant.
\end{remark}

\begin{definition}
	Let $(G,<_G,P_G)$ and $(H,<_H,P_H)$ be bi-ordered groups, and
	let $f\from G \to H$ be a surjective homomorphism.
	The bi-order of $H$ is \emph{induced by} the bi-order of $G$ if either of the following equivalent conditions is true:
	\begin{enumerate}
		\item For all $a,b\in G$, $a<_G b$ implies $f(a)\leq_H f(b)$,
		\item $f(P_G-\ker f)=P_H$. \qedhere
	\end{enumerate}
\end{definition}
\begin{remark}
\label{Rem:ManyToOne}
	Given a surjective homomorphism $f\from G \to H$, each bi-ordering of $G$ can induce either zero or one bi-ordering of $H$.
	However, a given bi-ordering of $H$ may be induced by many distinct bi-orderings of $G$.
\end{remark}

The following lemma constructs one such bi-order of $G$.

\begin{lemma}
	\label{Lem:constructpositivecone}
	Suppose $f\from G \to H$ is a surjective homomorphism of groups.
	Let $P_H$ be a positive cone of $H$ and let $P_K$ be a positive cone of $K = \ker f$ invariant under conjugation by elements of $G$.
	The set
	\[
		P=\{g\in G : f(g)\in P_H \text{ or } g \in P_K\}
	\]
	is a positive cone of $G$ which induces $P_H$.
\end{lemma}

\begin{proof}
	By construction, $f(P-\ker f)\subset P_H$,
	and since $f$ is surjective and $1\notin P_H$, we have $P_H\subset f(P-\ker f)$.
	Therefore $f(P-\ker f)= P_H$.
	It remains to show that $P$ is a positive cone. We check the properties in \refdef{PositiveCone} as follows.

	\begin{enumerate}	
	\item Consider $x,y\in P$.
	If $x,y\in P_K$ then $xy\in P_K$ so $xy\in P$.
	If $f(x)\in P_H$ then either $f(xy)=f(x)f(y)$ is a product of elements in $P_H$ or $f(xy)=f(x)$.
	Therefore, $xy\in P$.
	Similarly, if $f(y)\in P_H$ then $xy\in P$.
	This proves \refdef{PositiveCone}\refitm{Closed}.

	\item Since $P_H$ is a positive cone, $f(1)=1\notin P_H$.
	Since $P_K$ is a positive cone, $1\notin P_K$.
	Therefore, $1\notin P$.
	
	Consider $x\in P$.
	We will show that $x^{-1}\notin P$.
	Suppose $f(x)\in P_H$.
	Then, $f(x^{-1})\notin P_H$, and
	since $f(x)\neq 1$, $x$ is not in $\ker f$ so $x^{-1}\notin P_K$.
	Therefore, $x^{-1}\notin P$.
	
	Suppose $x\in P_K$.
	Then $x^{-1}\notin P_K$ since $P_K$ is a positive cone.
	Since $x\in K$ and $P_H$ is a positive cone, $f(x^{-1})=1\notin P_H$.
	Therefore, $x^{-1}\notin P$.

	Since $1\notin P$ and $P\cap P^{-1}=\emptyset$, it follows that $G = P \sqcup P^{-1} \sqcup \{1\}$.
	This proves \refdef{PositiveCone}\refitm{Partition}.
    
    \item  Consider $x\in P$, and
	suppose $g\in G$.
	If $f(x)\in P_H$,
	then $f(g^{-1}xg)=f(g)^{-1}f(x)g(x)\in P_H$ since $P_H$ is a positive cone.
	If $x\in P_K$,
	then $g^{-1}xg\in P_K$ since $P_K$ is invariant under conjugation by elements of $G$.
	Therefore, $g^{-1}xg\in P$ and we have \refdef{PositiveCone}\refitm{Invariant}. \qedhere
	\end{enumerate}
\end{proof}

An application of \reflem{constructpositivecone} is the following result on semidirect products of groups with the integers.

\begin{proposition}
	\label{Prop:BOInvariantOrder}
	Let $G$ be a group and let $\phi\in\aut(G)$.
	A semidirect product $G\rtimes_{\phi}\ZZ$ is bi-orderable if and only if there is a bi-order of $G$ invariant under $\phi$.
\end{proposition}

\begin{proof}
	First suppose that $G\rtimes_{\phi}\ZZ$ has a bi-order $<$.
	The bi-order $<$ restricts to a bi-order $<_G$ on $G$.
	Since the action of $\phi$ is conjugation in $G\rtimes_{\phi}\ZZ$ it preserves $<_G$.
	
	In the other direction, suppose that there is a bi-order of $G$ with positive cone $P_K$ invariant under $\phi$.
	By an abuse of notation, we also denote the corresponding subset of $G\rtimes_{\phi}\ZZ$ as $P_K$. 
%
	We apply \reflem{constructpositivecone} with $G\rtimes_{\phi}\ZZ$ playing the role of $G$, the integers $\ZZ$ playing the role of $H$, and the projection map $f\from G\rtimes_{\phi}\ZZ\to\ZZ$ being the surjective homomorphism. For the positive cone $P_H$ on $\ZZ$ we take the set of positive integers.
	We check that $P_K$ is invariant under conjugation by elements of $G\rtimes_{\phi}\ZZ$ as follows.
	Consider an element $(x,0) \in P_K$ and an arbitrary element $(g,k)\in G\rtimes_{\phi}\ZZ$.
	We calculate
	\[
		(g,k)\cdot (x,0)\cdot(g,k)^{-1}=(g\phi^k(x)g^{-1},0)
	\]
	Since $P_K$ is invariant under $\phi$ and conjugation in $G$, we have that $(g\phi^k(x)g^{-1}, 0)\in P_K$.
	Thus $P_K$ is invariant under conjugation by elements of $G\rtimes_{\phi}\ZZ$.
	Applying \reflem{constructpositivecone}, we get that 
	\[
		P=\{(g,n)\in G\rtimes_{\phi}\ZZ \mid 0<n\text{ or } (n = 0 \text{ and } g\in P_K)\}  
	\]
	is a positive cone of $G\rtimes_{\phi}\ZZ$.
	Therefore, $G\rtimes_{\phi}\ZZ$ is bi-orderable.
\end{proof}

%
%

\subsection{Convex Subgroups}

Given a surjective homomorphism of bi-orderable groups $f\from G \to H$, we now consider when a bi-ordering of $G$ induces a bi-ordering of $H$.
The answer is related to the convex subgroups of $G$;
see \refdef{Convex}.
We first restate the definition of convexity in terms of positive cones.

\begin{proposition}
\label{Prop:CosetOrder}
    Let $(G,<,P)$ be a bi-ordered group.
    A subgroup $C\subset G$ is convex with respect to $<$ if and only if for each $g \in G$, either $gC\subset P$, $gC\subset P^{-1}$, or $gC=C$.
    The analogous statement for right cosets also holds.
\end{proposition}

\begin{proof}
    Suppose that $C$ is convex with respect to $<$.
    Let $gC$ be a left coset of $C$ such that $gC\neq C$.
    Thus, $g\notin C$.
    Let $h = gc$ where $c\in C$.
    For a contradiction, assume that $g\in P$ and $h\in P^{-1}$.
    Therefore $h=gc<1<g$.
    Multiplying by $g^{-1}$ on the left we get $c<g^{-1}<1$.
    Since $C$ is convex, this implies that $g\in C$ and we have reached a contradiction.
    Therefore, either $gC\subset P$ or $gC\subset P^{-1}$.
    
    Now suppose that for each $g \in G$, either $gC\subset P$, $gC\subset P^{-1}$, or $gC=C$.
    Let $a$ and $b$ be in $C$, and
    let $a<g<b$ for some $g\in G$.
    These inequalities imply that $1<ga^{-1}$ and $gb^{-1}<1$.
    Thus, $gC$ is not a subset of $P$ or $P^{-1}$.
    Therefore we have that $gC=C$ and so $g\in C$.
    
    The proof for right cosets is similar.
\end{proof}


As a consequence of \refprop{CosetOrder}, we have the following.
\begin{corollary}
\label{Cor:ConvexQuotientBiorder}
Let $(G,<,P)$ be a bi-ordered group.
If $C$ is a normal convex subgroup of $G$ then there is a bi-ordering on $G/C$ whose positive cone is the set of cosets $gC$ such that $gC \subset P$.
This bi-ordering of $G/C$ is the one induced by $G$ using the quotient map. \qed
\end{corollary}


%
%

We now state two useful results about convex subgroups.
We refer the reader to a paper by Conrad~\cite{Con59} for proofs of these facts in the context of right-orders.
However, the arguments work equally well for bi-orders.

\begin{proposition}[Conrad~\cite{Con59}, Section~3.3]
	\label{Prop:NestedConvex}
	Let $G$ be a bi-ordered group the set of all convex subgroups of $G$ is totally ordered by
inclusion. \qed
\end{proposition}

\begin{remark}
    Recall that a convex subgroup $C$ of a bi-ordered group $(G,<)$ is \emph{maximal} if $C$ is proper and the only convex subgroups containing $C$ are $G$ and $C$.
	By Proposition \ref{Prop:NestedConvex}, maximal convex subgroups are unique when they exist.
\end{remark}

\begin{proposition}[Conrad~\cite{Con59}, Section~3.6]
\label{Prop:ConvexInduced}
	Let $(G,<)$ be a bi-ordered group.
	Suppose that $f\from G \to H$ is a surjective homomorphism of groups.
	If the bi-ordering $<$ induces a bi-ordering on $H$ then $\ker f$ is convex with respect to $<$.
	In particular, $H$ is bi-orderable if and only if $\ker f$ is convex with respect to some bi-ordering of $G$. \qed
\end{proposition}

\begin{proposition}[Clay and Rolfsen~\cite{ClaRol12}, Lemma 2.4]
	\label{Prop:FiniteGeneratorMaxConvex}
	If $(G, <)$ is a finitely-generated non-trivial bi-ordered group, then there exists a unique maximal convex subgroup $C$ of $G$ satisfying $C\neq G$.
	Moreover, $C$ is normal in $G$ and $G/C$ is abelian.
	If an automorphism $\phi\from G\to G$ preserves $<$ then $\phi(C)=C$. \qed
\end{proposition}


Given a bi-ordered group $(G,<)$, knowing information about the set of subgroups convex with respect to $<$ can provide important information about the bi-ordering $<$.
The following theorem is a classical result of H\"older.
(See Clay and Rolfsen~\cite[Theorem 2.6]{ClayRolf16} for a more recent exposition of this result.)

\begin{theorem}[H\"older~\cite{Holder1901}]
\label{Thm:Holder}
Suppose that $(G,<)$ is a bi-ordered group with no proper non-trivial convex subgroups.
Then there is an injective homomorphism $f\from G\to \RR$ such that the bi-order of the image of $f$ induced by $<$ is the same as the one induced by the usual ordering of $\RR$.
\end{theorem}




\subsection{Constructing bi-orders}

In this section we describe a general procedure for building bi-orders.
Let $G$ be a group, and consider a surjective homomorphism $f\from G\to H$.
Let $C$ be the kernel of $f$.
We can define a positive cone on $G$ by specifying a positive cone of $H$ and a positive cone on $C$ invariant under conjugation by elements of $G$ as in \reflem{constructpositivecone}.
By \refprop{ConvexInduced}, the subgroup $C$ will be convex with respect to this bi-order.

Abelian groups are a convenient choice for $H$ since bi-orderability is well understood for these groups.
In particular, Levi proves the following.
\begin{theorem}[Levi~\cite{Levi42}, Section~3]
	\label{Thm:Levi}
	An abelian group is bi-orderable if and only if it is torsion-free.
	Thus any free abelian group is bi-orderable. \qed
\end{theorem}

\begin{example}
\label{Exa:OrderZxZ}
	Suppose that $A \cong \ZZ^2$.
	\begin{enumerate}
	\item
	One way to define a bi-ordering of $A$ is to compare elements one factor at a time.
	In other words, an element $(m,n)\in A$ is positive when $m>0$ or when $m=0$ and $n>0$.
	The set $\{0\}\times\ZZ$ is convex with respect to this bi-ordering.
	This type of bi-ordering is referred to as a \emph{lexicographical bi-ordering}.
	
	\item
	\label{Itm:IrrationalSlope}
	Another option is to
	let $v \in \RR \oplus \RR$ be any vector with irrational slope.
	Then (with the dot product defined as usual) the set $P = \{ u \in A \mid u \cdot v > 0 \}$ is a positive cone on $A$.
	This bi-ordering does not have any proper non-trivial convex subgroups. \qedhere
	\end{enumerate}
\end{example}

\begin{example} \label{Exa:InfiniteLex}
	Suppose that $A$ is a free abelian group with a countably infinite basis $\calB$.
	Choose a total ordering of $\calB$.
	For example, we could index the elements of $\calB$ by the integers or rational numbers.
	Every element $x\in A$ can be represented uniquely as a linear sum of elements
	\[
		x=\sum_{v\in\calB} k_v v
	\]
	where $\{k_{v}\}_{v\in\calB}$ is a set of integer coefficients, all but finite many of which are zero.
	For each $a$, let $v(a)$ be the maximal basis element in $\calB$ such that $a$ has a non-zero coefficient.
	Define $P$ to be the set of all elements $a\in A$ such that the coefficient of $v(a)$ is positive.
\end{example}

We will now construct bi-orderings of a group by choosing a sequence of nested normal subgroups in which every quotient of consecutive terms in a free abelian group.
\refthm{Levi} ensures that our quotients have bi-orders.
We can then use the bi-orders of the quotients to define a bi-ordering of the entire group.

\begin{construction}
\label{Con:BiorderConstructionSchemeFin} 
Let $G$ be a group.
Suppose that we have a nested sequence of normal subgroups of $G$
\[
	G=C_1\triangleright C_2\triangleright \cdots \triangleright C_n
\]
where $C_i/C_{i+1}$ is a free abelian group for $1\leq i <n$.
For each $1\leq i <n$, let $p_i\from C_i\to C_i/C_{i+1}$ be the quotient map.
Choose a positive cone of $C_n$ invariant under conjugation by all elements of $G$.
For every $0\leq i<n$, choose a positive cone $P_i$ for the quotient $C_i/C_{i+1}$ invariant under conjugation of elements of $G$ so for all $g\in C_i$ and $x\in G$, we have
$p_i(g)\in P_i$ if and only if $p_i(xgx^{-1})\in P_i$.
Define an element $g\in G$ to be positive when either of the following occurs:
\begin{enumerate}
	\item $g\in C_i-C_{i+1}$ and $p_i(g)\in P_i$ for some $1\leq i < n$
	\item $g\in P_n$ \qedhere
\end{enumerate}
\end{construction}

We can extend this construction to an infinite nested sequence of subgroups when their intersection is trivial.

\begin{construction}
\label{Con:BiorderConstructionSchemeInf}
Let $G$ be a group.
Suppose that we have a infinite nested sequence of normal subgroups of $G$
\[
	G=C_1\triangleright C_2\triangleright C_3\triangleright \cdots
\]
where $C_i/C_{i+1}$ is a free abelian group and $\bigcap C_i=\{1\}$.
For each $i$, let $p_i\from C_i\to C_i/C_{i+1}$ be the quotient map.
Choose a positive cone $P_i$ for the quotient $C_i/C_{i+1}$ invariant under conjugation by elements of $G$.
Define an element $g\in G$ to be positive when $g\in C_i-C_{i+1}$ and $p_i(g)\in P_i$ for some $i$.
\end{construction}

A key example of the use of this general strategy is in constructing bi-orders on free groups.

\begin{definition}
	Let $G$ be a finite rank free group.
	Let $G_1 = G$ and for $n>1$ let $G_n=[G_{n-1},G]$.
	The sequence  $\{G_n\}_{n\geq1}$ is the \emph{lower central series} of $G$.
	For $n\geq1$, let $A_n=G_n/G_{n+1}$ be the \emph{$n$th lower central series quotient}.
\end{definition}

\begin{remark} 
	Each $A_n$ is a free abelian group. 
	Hall's work~\cite{Hall34} implies that the rank of $A_n$ is less than or equal to $2^n$. 
\end{remark}

\begin{theorem}[Magnus~\cite{Magnus35}]
	\label{Thm:TrivialIntersection}
	The set of terms of the lower central series of a free group has trivial intersection. \qed
\end{theorem}

The following definition is due to Rolfsen~\cite[Appendix A]{KinRolf18}.

\begin{definition}
\label{Def:Standard}
	A \emph{standard bi-order} on a free group $G$ is one constructed as follows.
	Applying \refthm{Levi}, choose a positive cone $P_n$ on each lower central series quotient $A_n$.
	By \refthm{TrivialIntersection}, for each nontrivial element $g\in G$ there is a unique positive integer $n(g)$ such that $g\in G_{n(g)}-G_{n(g)+1}$.
	Let $[g]_n$ be the class of $g$ in $A_n$.
	We define a positive cone $P$ on $G$ by:
	\begin{equation}
    		\label{positivecone}
    		P:=\{g\in G \mid g \neq 1 \text{ and } [g]_{n(g)}\in P_{n(g)}\}
    		. \qedhere
	\end{equation}

\end{definition}

\begin{proposition}
	\label{Prop:StandardConvexDef}
    A bi-ordering of the free group is standard if and only if every term of the group's lower central series is convex with respect to that bi-order.
\end{proposition}

\begin{proof}
    Let $(G,<,P)$ be a bi-ordered free group.
    If $P$ is standard, then by \reflem{constructpositivecone} the restriction of the bi-order to $G_{n-1}$ induces the bi-ordering of $A_n$ given by $P_n$.
	Thus by \refprop{ConvexInduced}, the next term $G_n$ is convex in $G_{n-1}$.
    It follows by induction that $G_n$ is convex in $G$.
    
    Suppose that $G_n$ is convex for all $n$.
    Then the bi-order of $G$ induces a bi-order $<_{n}$ on $A_n:=G_n/G_{n+1}$ for each $n$.
    Define $P_n$ to be the positive cone of $(A_n,<_{n})$.
    $P$ can be described as the standard bi-order of $G$ determined by the positive cones $P_n$.
\end{proof}

\begin{proposition} \label{Prop:PreserveStandardOrders}
	Suppose that $G$ is a finite rank free group
	and $\phi,\phi'\in\aut(G)$ are automorphisms such that the induced maps on the abelianization of $G$ are the same.
	Then, $\phi$ preserves a standard bi-ordering of $G$ if and only if $\phi'$ does.
\end{proposition}

\begin{proof}
	Consider a standard positive cone $P$ of $G$ determined by a collection of positive cones $P_n$ of $A_n=G_n/G_{n+1}$.
	Given an automorphism $\psi\in\aut(G)$, denote by $\psi_n$ the induced automorphisms on $A_n$.
	Then $\psi$ preserves $P$ if and only if $\psi_n(P_n)=P_n$ for all positive integers $n$.
	By a lemma of Perron and Rolfsen \cite[Lemma 4.5]{PerRolf03},
	there are canonical embeddings $\iota_n\from A_n \hookrightarrow A_1^{\otimes n}$ such that the following diagram commutes.
	\[
  	  \begin{tikzcd}
    		A_n \arrow[hookrightarrow]{r}{\iota_n} \arrow[swap]{d}{\psi_n} & A_1^{\otimes n} \arrow[swap]{d}{\psi_1^{\otimes n}}\\
		A_n \arrow[hookrightarrow]{r}{\iota_n} & A_1^{\otimes n}
  	  \end{tikzcd}
	\]
	
	Consider two automorphisms $\phi$ and $\phi'$ of $G$ that induce the same automorphism of the abelianization of $G$.
	In other words, we have that $\phi_1=\phi_1'$.
	Suppose that $\phi(P)=P$.
	Then for each $n$, we have that $\phi_n(P_n)=P_n$
	Thus, it follows that $\phi_1^{\otimes n}(\iota_n(P_n))=\iota_n(P_n)$.
	Since $\phi_1$ is equal to $\phi_1'$, it follows that $\phi'_n(P_n)=P_n$ for each $n$.
	Therefore, $\phi'(P)=P$.
\end{proof}

\section{Punctured torus bundles}
\label{Sec:PuncTorusBundles}

\begin{definition}
We define the \emph{(untwisted) punctured torus bundle} $M$ over the circle with monodromy $h$ as
\[
	M \cong \frac{T \times [0,1]}{(x,1) \sim (h(x), 0)} 
\]
where $T$ is a torus with a boundary component
and $h$ is an automorphism of $T$ which fixes $\partial T$ pointwise.
\end{definition}

\begin{proposition}
\label{Prop:BiOrderablePuncturedTorusBundles}
Suppose that $M$ is a hyperbolic punctured torus bundle. Then $\pi_1(M)$ is bi-orderable if and only if $M$ is untwisted.
\end{proposition}

\begin{proof}
For hyperbolic punctured torus bundles the Alexander polynomial has two real roots.
These roots have the same sign, which is positive in the untwisted case and negative in the twisted case.
For untwisted punctured torus bundles then, $\pi_1(M)$ is bi-orderable by \refthm{PerRolf}.
Conversely, if  $\pi_1(M)$ is bi-orderable then by \refthm{ClayRolf} at least one of the roots is positive, and therefore the bundle is untwisted.
\end{proof}

Up to isotopy, the monodromy $h$ is a product of Dehn twists about two simple closed curves on $T$ which intersect once. By choosing a base point $x$ on $\partial T$, $h$ induces a well-defined automorphism $h_*$ of $\pi_1(T,x)$ which is a rank 2 free group.
The fundamental group of $M$ is given by
\[
    \pi_1(M) \cong G \rtimes_{h_*} \ZZ
\]

Here $G \cong \langle \alpha, \beta \rangle$ is the fundamental group of the punctured torus, generated by loops $\alpha$ and $\beta$ as shown in \reffig{PuncturedTorus}.  
Fixing notation, let $\ZZ \cong \langle \tau \rangle$.

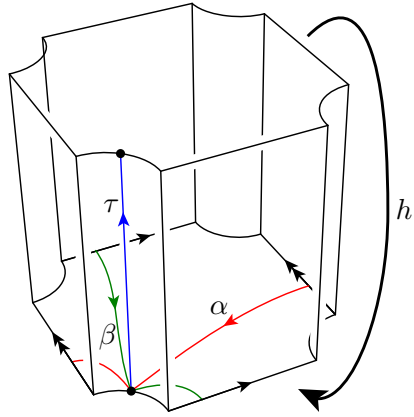
\begin{figure}[htbp]
\centering 
\labellist
\small\hair 2pt
\pinlabel $\tau$ [r] at 80 142
\pinlabel $\alpha$ [br] at 154 65
\pinlabel $\beta$ [r] at 78 52
\pinlabel $h$ [l] at 263 140
\endlabellist
\includegraphics[width = 0.4\textwidth]{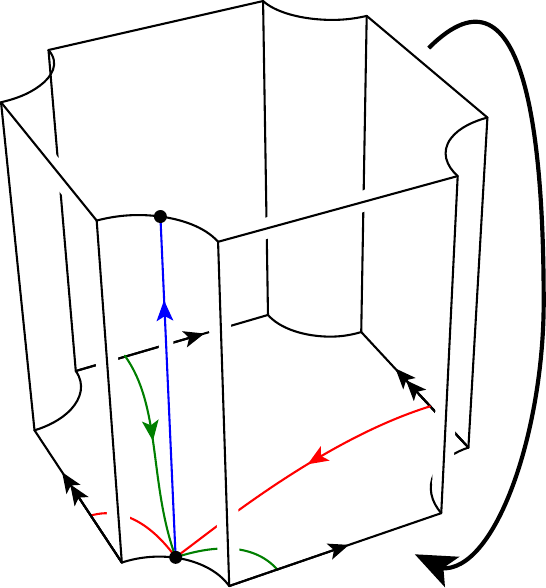}
\caption{Generators of the fundamental group for a punctured torus bundle.} 
\label{Fig:PuncturedTorus}
\end{figure}

\begin{remark}
\label{Rem:Fig8Twists}
Consider the figure eight knot complement, with monodromy given by a single Dehn twist around a curve parallel to $\beta$ (as in \reffig{PuncturedTorus}), followed by a single Dehn twist around a curve parallel to $\alpha$.
One can check that the images of $\alpha$ and $\beta$ under the monodromy are $\tau^{-1}\alpha\tau = h(\alpha) = \alpha\beta\alpha$ and $\tau^{-1}\beta\tau = h(\beta) = \beta\alpha$. 
\end{remark}

We will use the following well-known fact a few times.
See Farb-Margalit \cite[Section~13.1]{FarMar12} for details.

\begin{lemma}
\label{Lem:Eigenvalues}
A pseudo-Anosov map $A \in \SL(2, \ZZ)$ cannot have eigenvalues of either $1$ or $-1$. \qed
\end{lemma}


\begin{lemma}
	\label{Lem:H1Bundle}
	$\rk H_1(M;\ZZ)=1$
\end{lemma}

\begin{proof}
	Let $[\alpha]$ and $[\beta]$ be the integral homology classes of the curves $\alpha$ and $\beta$, as in \reffig{PuncturedTorus}.
	Let $A\in\SL(2,\ZZ)$ be the map $h_+$ on homology with respect to the basis $[\alpha],[\beta]$.
	The homology group $H_1(M;\ZZ)$ decomposes as $H_1(M;\ZZ) \cong M_T \oplus \langle t\rangle$ where
	\[
		M_T=\langle [\alpha],[\beta] : [\alpha]=h_+[\alpha],[\beta]=h_+[\beta]\rangle .
	\]
	$M_T$ is torsion if and only if $\det(I-A)\neq 0$.
	Since $h$ is pseudo-Anosov, 1 cannot be an eigenvalue of $A$ by \reflem{Eigenvalues}.
	Therefore $\det(I-A)=\ch_A(1)\neq 0$.
	Therefore, $M_T$ is torsion and $\rk H_1(M;\ZZ)=1$.
\end{proof}

\section{Maximal convex subgroups} 
\label{Sec:G2MaxConvex}

\subsection{Proof of \refthm{toporder}}
\label{Sec:ProofOfTopOrder}
Recall that $G$ is the free group on two elements.
The following proposition tells us that any pair of bi-orders on a group $\pi_1(M) \cong G \rtimes_h \ZZ$ either agree or precisely disagree at the level of the $\ZZ$ term.
(This is not the case for a general semidirect product of a group with the integers.)

\begin{proposition}
\label{Prop:MaxConvexPi1M}
    The maximal convex subgroup of any bi-ordering of $\pi_1(M)$ is $G$, the fundamental group of the punctured torus.
\end{proposition}

\begin{proof}
	Let $<$ be a bi-order of $\pi_1(M)$.
	Since $\pi_1(M)$ is finitely generated, it has a maximal convex subgroup $C$ by Proposition \ref{Prop:FiniteGeneratorMaxConvex}.

	By \refcor{ConvexQuotientBiorder}, $H:=\pi_1(M)/C$ is a nontrivial bi-orderable subgroup.
	By \refthm{Holder}, $H$ is abelian.
	Thus, $H$ is isomorphic to a nontrivial bi-orderable quotient of $H_1(M;\ZZ)$.
	By Lemma \ref{Lem:H1Bundle}, we have that $\rk H_1(M;\ZZ) = 1$ and $H\cong\ZZ$.
	Thus $\pi_1(M)/C \cong \ZZ$.
	Also $\pi_1(M) \cong G \rtimes_h \ZZ$, so both $C$ and $G$ are kernels of surjections from $\pi_1(M)$ to $\ZZ$.
	But because $\rk H_1(M;\ZZ) = 1$, there is a unique (up to sign) surjection from $\pi_1(M)$ to $\ZZ$.
	Therefore $C=G$.
\end{proof}

Thus a bi-ordering of $\pi_1(M) \cong G \rtimes_h \ZZ$ is determined by a bi-ordering of the free group $G$ (together with a choice of which way we order $\ZZ$).
The next proposition tells us that when bi-ordering the free group $G$ subject to invariance under $h$, we initially  have no choice but to follow the standard construction, as given in \refdef{Standard}.

\begin{proposition}
\label{Prop:MaxConvexG}
	Consider any bi-order $<$ on $\pi_1(M)$ and the induced bi-order $<_G$ on $G$.
    The maximal convex subgroup of $G$ with respect to $<_G$ is $G_2=[G,G]$.
\end{proposition}

\begin{proof}
	Let $<$ be a bi-order of $\pi_1(M)$.
	By Proposition \ref{Prop:BOInvariantOrder}, this induces a bi-order $<_G$ on $G$ which is invariant under $h$.
	Since $G$ is finitely generated, it has a maximal convex subgroup $C$ by Proposition \ref{Prop:FiniteGeneratorMaxConvex}.

	By \refthm{Holder}, $H:=G/C$ is a nontrivial abelian bi-orderable subgroup.
	Thus, $H$ is a nontrivial bi-orderable quotient of $H_1(T;\ZZ)$
	so either $H\cong\ZZ$ or $H\cong\ZZ^2$.
	
	Assume for a contradiction that $H\cong\ZZ$.
	Let $p \from G\to H_1(T;\ZZ)$ be the abelianization map.
	The monodromy $h$ induces a map $h_+$ on $H_1(T;\ZZ)$.
	By \refprop{FiniteGeneratorMaxConvex}, $h(C)=C$ so $h_+(p(C))=p(h(C))=p(C)$.
	Since $H\cong\ZZ$ then $p(C)\cong\ZZ$,
	and since $h_+(p(C))=p(C)$ the eigenvalues are either $1$ or $-1$. 
	This is then a contradiction to \reflem{Eigenvalues}.
	Therefore, $H\cong\ZZ^2$.
	It follows that $C=G_2$.
\end{proof}

\begin{proof}[Proof of \refthm{toporder}]
This follows from Propositions~\ref{Prop:MaxConvexPi1M} and~\ref{Prop:MaxConvexG}.
\end{proof}

\subsection{Monodromy action on $G_2$}

Consider the abelian cover $\pi \from \hat{T} \to T$ corresponding to the subgroup $G_2$.
We orient the components of $\partial \hat{T}$ counterclockwise.
Choose a component of $\partial \hat{T}$ and call it $\gamma_0$.
Let $\hat{x}$ be the lift of the basepoint $x$ on $\gamma_0$.
Elements of $G$ can be lifted to homotopy classes of paths in $\hat{T}$ based at $\hat{x}$.

\begin{definition}
\label{Def:hHat}
Let $\hat{h}\from \hat{T}\to \hat{T}$ be the lift of $h$ that preserves $\gamma_0$. 
Let $\hat{h}_+$ be the map induced on $H_1(\hat{T};\ZZ)$ by $\hat{h}$.
\end{definition}

\reffig{ActionOfh} illustrates $\hat{h}$ in the example of the figure eight knot complement. 

\begin{figure}[htbp]
\subfloat[The lift of the commutator $\alpha\beta\alpha^{-1}\beta^{-1}$ is homotopic to a loop $\gamma_0$ in $\hat{T}$. ]{
\centering 
\labellist
\small\hair 2pt
\pinlabel $\gamma_0$ [tr] at 150 150
\pinlabel $\hat{x}$ [bl] at 162 162
\endlabellist
\includegraphics[height = 5cm]{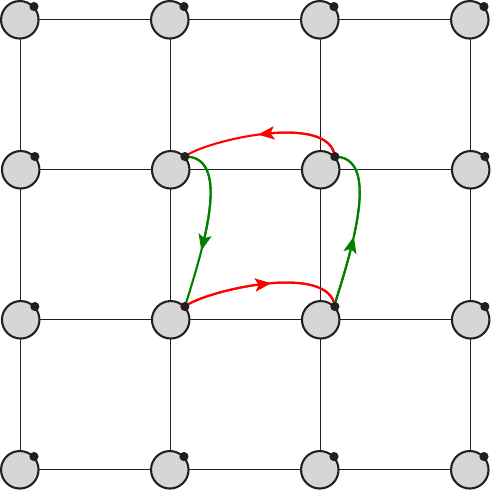}
}
\quad
\subfloat[The image $\hat{h}(\gamma_0)$ is homotopic to $\gamma_0$. ]{
\centering 
\labellist
\small\hair 2pt
\pinlabel $\gamma_0$ [tr] at 222 150
\pinlabel $\hat{x}$ [bl] at 234 162
\endlabellist
\includegraphics[height = 5cm]{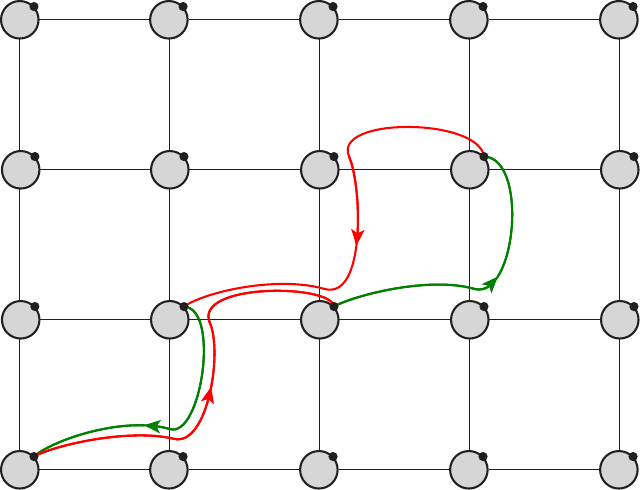}
}

\caption{The map $\hat{h}$ acts on the abelian cover $\hat{T}$, and preserves $\gamma_0$. Here we show the image for the figure eight knot complement. By \refrem{Fig8Twists}, $h(\alpha) = \alpha\beta\alpha$ and $h(\beta) = \beta\alpha$. Note that $\gamma_0$ is fixed by $\hat{h}$.} 
\label{Fig:ActionOfh}
\end{figure}

\begin{definition} \label{Def:Basis}
Let $\calB$ be the countable basis of $H_1(\hat{T},\ZZ)$  consisting of the homology classes corresponding to the components of $\partial \hat{T}$.
\end{definition}
\begin{notation}
To lighten the notation, we will use the symbol $\gamma$ for a generic component of $\partial \hat{T}$ and also for its homology class.
\end{notation}

Recall that $\pi$ is the covering map from $\hat{T}$ to $T$.

\begin{lemma}
\label{Lem:q}
Suppose that $\gamma$ is a component of $\partial \hat{T}$. 
Let $\eta$ be a path in $\hat{T}$ from $\gamma_0$ to $\gamma$.
Then the class $[\pi(\eta)] \in H_1(T; \ZZ)$ is independent of the choice of $\eta$.
\end{lemma}

\begin{proof}
Suppose that $\eta$ and $\eta'$ are two such paths. 
Then the concatenation $\eta' \cdot \eta^{-1}$ is a loop in $\hat{T}$.
Thus $\pi(\eta' \cdot \eta^{-1}) \in [G, G]$.
Therefore $0 = [\pi(\eta' \cdot \eta^{-1})] =   [\pi(\eta') \cdot \pi(\eta^{-1})] =  [\pi(\eta') \cdot \pi(\eta)^{-1}] =  [\pi(\eta')] - [\pi(\eta)]$.
Thus $[\pi(\eta)] = [\pi(\eta')]$. 
\end{proof}

We may therefore make the following definition.
\begin{definition}
\label{Def:q}
For any choice of $\eta$ as in \reflem{q},
we define 
\[
q:\mathcal{B}\to H_1(T;\ZZ) \cong \ZZ^2 \subset \RR^2 \quad \mbox{by}\quad q(\gamma)= [\pi(\eta)].\qedhere
\]
\end{definition}

\begin{lemma}
The map $q$ is a bijection.
\end{lemma}

\begin{proof}
Recall that the curves $\alpha$ and $\beta$ are our generators for $G$, the fundamental group of the punctured torus $T$,  as shown in \reffig{PuncturedTorus}.
For any element $m[\alpha] +n[\beta]$ of $H_1(T;\ZZ)$, if we lift the loop $\alpha^m\cdot\beta^n$ to $\hat{T}$ we find a path based at $\hat{x}$ which ends on a loop $\gamma$ such that $q(\gamma) = m[\alpha] +n[\beta]$.
Thus $q$ is surjective.

Now suppose that $\gamma$ and $\gamma'$ are elements of $\bdy \hat{T}$ with $q(\gamma) = q(\gamma')$.
Let $\eta$ and $\eta'$ be paths for $\gamma$ and $\gamma'$ as in \reflem{q}.
So $[\pi(\eta)] = [\pi(\eta')]$.
Following the calculation in the proof of \reflem{q} backwards, we deduce that $\pi(\eta' \cdot \eta^{-1}) \in [G, G]$. 
This implies that $\eta' \cdot \eta^{-1}$ is a loop in $\hat{T}$, and so $\gamma = \gamma'$.
Thus $q$ is injective.
\end{proof}

\begin{lemma}
\label{Lem:ActionOfHatH+}
We have the following.
\begin{enumerate}
\item 
\label{Itm:Contained}
$\hat{h}_+(\calB) \subset \calB$
\item 
\label{Itm:Commutes}
$q \circ \hat{h}_+ = h_+ \circ q$
\end{enumerate}
\end{lemma}

\begin{proof}
Since $\hat{h}$ is a self-homeomorphism of $\hat{T}$, it sends boundary loops to boundary loops.  
Thus $\gamma \in \calB$ implies that $\hat{h}_+(\gamma) \in \calB$ and we have \refitm{Contained}.

Recall that $\pi\from \hat{T} \to T$ is the covering map.  
Then by definition, $h \circ \pi = \pi \circ \hat{h}$.
Let $\gamma \in \calB$ and let $\eta$ be a path in $\hat{T}$ from $\gamma_0$ to $\gamma$. Then
\[
h_+ \circ q(\gamma) = h_+([\pi(\eta)]) = [h \circ \pi(\eta)] = [\pi (\hat{h}(\eta))]
\]
Note that $\hat{h}(\eta)$ is a path based at $\hat{x}$ which ends on $\hat{h}(\gamma)$.
So we get that $ [\pi (\hat{h}(\eta))] = q([\hat{h}(\gamma)]) = q \circ \hat{h}_+[(\gamma)]$ and we have proved \refitm{Commutes}.
\end{proof}

\begin{definition} 
\label{Def:Shift}
Let $\gamma \in \calB$. The \emph{shift} of $\gamma$ by $(m,n)$ is 
\[
s_{(m,n)}(\gamma) = q^{-1}\big((m,n) + q(\gamma)\big). \qedhere
\] 
\end{definition}

\begin{definition}
Let $p_1:G\to\ZZ^2$ be the abelianization map.
Let $p_2:G_2\to G_2/[G_2,G_2]$ be the quotient map.
\end{definition}

\begin{lemma} 
	\label{Lem:Shift}
	Suppose that $g \in G_2$ and $x \in G$. 
	Suppose that the homology class $p_1(x)$ in $H_1(T,\ZZ)$ is  $m[\alpha]+n[\beta]$
	and that the homology class $p_2(g)$ in $H_1(\hat{T},\ZZ)$ is
	\[
		\sum_{\gamma\in\calB} k_\gamma \gamma.  
	\]
	Then the homology class of $p_2(x g x^{-1})$ in $H_1(\hat{T},\ZZ)$ is 
	\[
		\sum_{\gamma\in\calB} k_\gamma s_{(m,n)}(\gamma)
	\]
	where $s_{(m,n)}(\gamma)$ is the shift of $\gamma$ by $(m,n)$, as in \refdef{Shift}.
\end{lemma}
	
\begin{proof}
	First note that the set 
	$
	\{ y \alpha \beta \alpha^{-1} \beta^{-1} y^{-1} \mid y \in G \}
	$
	generates $G_2$.
	Therefore we can write 
	\[
	    g = \prod_i y_i (\alpha \beta \alpha^{-1} \beta^{-1})^{\pm1} y_i^{-1}
	\]
	for some finite collection of $y_i \in G$.
	For any $y \in G$, we have that the element $\gamma = p_2 (y \alpha \beta \alpha^{-1} \beta^{-1} y^{-1})$ is a member of $\calB$.	
	Thus we can write 
	\[
	    p_2(g) = \sum_i p_2( y_i (\alpha \beta \alpha^{-1} \beta^{-1})^{\pm1} y_i^{-1}) = \sum_{\gamma\in\calB} k_\gamma \gamma
	\]
         where $k_\gamma$ is found by summing the $\pm1$ exponents of the terms in which the path $y_i$ connects $\gamma_0$ to $\gamma$ in $\hat{T}$.
	Conjugating by $x$, we can write
	\[
	    x g x^{-1} =  \prod_i x y_i (\alpha \beta \alpha^{-1} \beta^{-1})^{\pm1} y_i^{-1} x^{-1}
	\]
	Recall that the homology class $p_1(x)$ is $m[\alpha]+n[\beta]\in H_1(T,\ZZ)$.
	It follows that if $p_2 (y \alpha \beta \alpha^{-1} \beta^{-1} y^{-1}) = \gamma$ then
	$p_2 (x y \alpha \beta \alpha^{-1} \beta^{-1} y^{-1} x^{-1})$ is the shift $s_{(m,n)}(\gamma)$.
	Thus 
	\[
		p_2(xgx^{-1}) = \sum_i p_2( x y_i (\alpha \beta \alpha^{-1} \beta^{-1})^{\pm1} y_i^{-1} x^{-1}) = \sum_{\gamma\in\calB} k_\gamma s_{(m,n)}(\gamma). \qedhere
	\]
\end{proof}

\subsection{Proof of \refthm{G3Convex}}
\label{Sec:ProofOfG3Convex}

We will prove \refthm{G3Convex} through a series of lemmas.
First note that since $C$ is convex, $<_2$ induces a bi-ordering of $G_2/C$.
By \refthm{Holder}, there is an injective homomorphism $f\from G_2/C\to\RR$ respecting the bi-orders.
Also, since $G_2/C$ is abelian the quotient map factors through quotient maps $p_2\from G_2 \to H_1(\hat{T};\ZZ)\cong G_2/[G_2,G_2]$ and $\psi\from H_1(\hat{T};\ZZ) \to G_2/C$.
Define $\mu\from H_1(\hat{T};\ZZ)\to\RR$ to be $f\circ\psi$ as in the following diagram.

\[
  \begin{tikzcd}
    G_2 \arrow{r}{p_2}  & H_1(\hat{T};\ZZ) \arrow[swap]{d}{\psi} \arrow{dr}{\mu} \\
	& G_2/C \arrow[hookrightarrow,swap]{r}{f} & \RR
  \end{tikzcd}
\]

Let $P$ be the positive cone associated to $<_2$.
Let $Q^+=\mu^{-1}(\RR^+)$ and $Q^-=\mu^{-1}(\RR^-)$.
Since $\mu\circ p_2$ and $<_2$ induce the bi-ordering of the image of $\mu$, we have
\[
Q^+=p_2(P-C), \quad Q^-=p_2(P^{-1}-C), \quad \mbox{and}\quad \ker\mu = p_2(C).
\]

\begin{lemma}\label{Lem:PreserveSign}
	Suppose that $\phi\in\aut(G_2)$ preserves $C$ and the bi-ordering $(<_2,P)$.
	Let $\phi_+\in\aut(H_1(\hat{T};\ZZ))$ be the automorphism induced by $\phi$.
	Then the sets $Q^+$, $Q^-$, and $\ker\mu$ are invariant under $\phi_+$.
\end{lemma}

\begin{proof}
	By assumption $\phi$ preserves $P$ and $C$.
	Therefore, by the definition of an induced map, $\phi_+$ preserves $p_2(P-C)$, $p_2(P^{-1}-C)$, and $p_2(C)$.
\end{proof}

In other words, for all $\Homv\in H_1(\hat{T};\ZZ)$, the sign of $\mu(\Homv)$ is the same as the sign of $\mu(\phi_+(\Homv))$.

\begin{lemma}\label{Lem:PreserveRatios}
	Let $\Homv,\Homw\in H_1(\hat{T};\ZZ)$ with $\mu(\Homv)\neq 0$ and $\mu(\Homw)\neq 0$.
	Suppose that $\phi\in\aut(G_2)$ preserves $C$ and the bi-ordering $<_2$.
	Let $\phi_+\in\aut(H_1(\hat{T};\ZZ))$ be the automorphism induced by $\phi$.
	Then
	\[
		\frac{\mu(\Homv)}{\mu(\Homw)}=\frac{\mu(\phi_+(\Homv))}{\mu(\phi_+(\Homw))}.
	\]
\end{lemma}

\begin{proof}
	Let $d=\mu(\Homv)/\mu(\Homw)$, and define the set $\QQ_{>d}$ as follows.
	\[
		\QQ_{>d}=\{r\in\QQ\mid r>d\}
	\]
	Thus, $d=\inf \QQ_{>d}$.
	Consider an arbitrary element of $\QQ_{>d}$. 
	Write this element as $a/b$ where $a,b$ are integers with $b>0$. 
	Thus $d < a/b$.
	We first assume that $\mu(\Homw) > 0$.
	Then we have that
	\[
		0=d\mu(\Homw)-\mu(\Homv) < (a/b)\mu(\Homw)-\mu(\Homv) =  a\mu(\Homw)-b\mu(\Homv)
	\]
	\[
		=\mu\big(a\Homw-b\Homv\big).
	\]
	
	By \reflem{PreserveSign},
	\[
		0<\mu\big(\phi_+(a\Homw)-b\Homv)\big)=a\mu\big(\phi_+(\Homw)\big)-b\mu\big(\phi_+(\Homv)\big)
	\]
Thus,
	\[
		\frac{\mu(\phi_+(\Homv))}{\mu(\phi_+(\Homw))}<\frac{a}{b}
	\]
	Since $a/b$ was arbitrary,  $\mu(\phi_+(\Homv))/\mu(\phi_+(\Homw))$ is a lower bound for $\QQ_{>d}$.
	Therefore,
	\[
		\frac{\mu(\phi_+(\Homv))}{\mu(\phi_+(\Homw))}\leq d
	\]
	Similarly, $\mu(\phi_+(\Homv))/\mu(\phi_+(\Homw))$ is an upper bound for
	\[
		 \QQ_{<d}=\{r\in\QQ\mid r<d\}
	\]
	so
	\[
		\frac{\mu(\phi_+(\Homv))}{\mu(\phi_+(\Homw))}=d=\frac{\mu(\Homv)}{\mu(\Homw)}.
	\]
	A similar argument goes through when $\mu(\Homw) < 0$.
\end{proof}

The map $\mu$ is determined by its value on the basis $\calB$.
For each pair of integers $m,n$, define $\gamma_{m,n}$ to be $q^{-1}(m,n)$, where $q$ is as given in \refdef{q}.
Thus,
\[
	\gamma_{m,n}=s_{m,n}(\gamma_{0,0})
\]

\begin{lemma} \label{Lem:NuExponential}
	For all $m,n\in\ZZ$,
	\[
		\mu(\gamma_{m,n})=\mu(\gamma_{0,0})a^mb^n
	\]
	for some $a,b\in\RR^+$.
\end{lemma}

\begin{proof}
	For each pair of integers $m,n$,
	\[
		\gamma_{m,n}=p_2(\alpha^m\beta^n[\alpha,\beta]\beta^{-n}\alpha^{-m}]).
	\]
	Since conjugation preserves every positive cone of the free group $G$,
	\reflem{PreserveSign} implies that either $\mu(\calB)\subset\RR^+$, $\mu(\calB)\subset\RR^-$, or $\mu(\calB)=\{0\}$.
	Since $\mu$ is not a trivial map, $\mu(\gamma_{m,n})\neq 0$ for all integer pairs $m$ and $n$.
	
	Let $a$ and $b$ be defined as follows.
	\[
		a=\frac{\mu(\gamma_{1,0})}{\mu(\gamma_{0,0})}\quad \mbox{and}\quad b=\frac{\mu(\gamma_{0,1})}{\mu(\gamma_{0,0})}
	\]
	By \reflem{PreserveRatios},
	\[
		a=\frac{\mu(\gamma_{m+1,n})}{\mu(\gamma_{m,n})}\quad \mbox{and}\quad b=\frac{\mu(\gamma_{m,n+1})}{\mu(\gamma_{m,n})}
	\]
	for all $m$ and $n$.
	It follows that $\mu(\gamma_{m,n})=\mu(\gamma_{0,0})a^mb^n$.	
\end{proof}

The following lemma is proved using \emph{basic commutators}, see \cite[Chapter~3]{ClementMajewicsZyman17} for the details.

\begin{lemma}
\label{Lem:G2G3Z}
$G_2/G_3\cong\ZZ$.
\end{lemma}

\begin{proof}
By Theorem 3.1 of Clement-Majewics-Zyman \cite{ClementMajewicsZyman17}, any element of $G_2/G_3$ is generated modulo $G_3$ by the basic commutators of weight two.
In our case there is only one commutator of weight two, namely $[\beta, \alpha]$.
The result follows.
\end{proof}

\begin{lemma} \label{Lem:ConstantG3}
	If $\mu$ is constant on $\calB$,
	then $C=G_3$.
\end{lemma}

\begin{proof}
	Suppose that $\mu$ is constant on $\calB$.
	First we show that $G_3\subset C$.
	By the first isomorphism theorem, $C=\ker(\phi \circ p_2) = \ker \mu\circ p_2$.
	It is sufficient to show that for arbitrary elements $g\in G_2$ and $x\in G$, the element $gxg^{-1}x^{-1} \in G_3$ is in $\ker \mu\circ p_2$.
	To show this, suppose that the homology class $p_1(x)$ in $H_1(T,\ZZ)$ is  $m[\alpha]+n[\beta]$
	and that the homology class $p_2(g)$ in $H_1(\hat{T},\ZZ)$ is
	\[
		\sum_{\gamma\in\calB} k_\gamma \gamma.  
	\]
	By \reflem{Shift},
	\[
		p_2(g x g^{-1} x^{-1})=
		\sum_{\gamma\in\calB} k_\gamma \gamma - \sum_{\gamma\in\calB} k_\gamma s_{(m,n)}(\gamma)
		=\sum_{\gamma\in\calB} k_\gamma (\gamma - s_{(m,n)}(\gamma)).
	\]
	Since $\mu$ is constant on $\calB$, we have that
	\[
		\mu(p_2(g x g^{-1} x^{-1}))
		=\sum_{\gamma\in\calB} k_\gamma \big[\mu(\gamma) - \mu(s_{(m,n)}(\gamma))\big]=0 .
	\]
	Thus, $G_3\subset C$.
	By the third isomorphism theorem,
	\[
		\frac{G_2/G_3}{C/G_3}\cong \frac{G_2}{C}.
	\]
	By \reflem{G2G3Z}, we have that $G_2/C$ is a non-trivial bi-orderable quotient of $\ZZ$.
	Since $\ZZ$ is the only bi-orderable cyclic group, $C/G_3$ must be trivial.
	Therefore, $C=G_3$.
\end{proof}

Recall that from \refdef{hHat}, we lift the monodromy $h$ to a map $\hat{h}$ on the cover $\hat{T}$. Then $\hat{h}_+$ is the induced map on homology.

\begin{lemma} \label{Lem:NuOrbit}
	For all $\Homv\in H_1(\hat{T};\ZZ)$,
	\[
		\mu(\hat{h}_+(\Homv))=\mu(\Homv) .
	\]
\end{lemma}

\begin{proof}
	Let $\Homv\in H_1(\hat{T};\ZZ)$.
	The map $\hat{h}_+$ is induced by $h_*$, and
	by assumption, $h_*$ preserves $P$.
	By \reflem{PreserveSign}, $\mu(\Homv)=0$ if and only if $\mu(\hat{h}_+(\Homv))=0$.
	
	When $\mu(\Homv)\neq 0$,
	\begin{equation} \label{Eq:RationhPlus}
		\frac{\mu(\Homv)}{\mu(\gamma_{0,0})}=\frac{\mu(\hat{h}_+(\Homv))}{\mu(\hat{h}_+(\gamma_{0,0}))}
	\end{equation}
	by \reflem{PreserveRatios}.
	Since $\gamma_{0,0}$ is fixed by $\hat{h}_+$, we have $\mu(\gamma_{0,0})=\mu(\hat{h}_+(\gamma_{0,0}))$.
	Therefore, it follows from (\ref{Eq:RationhPlus}) that $\mu(\hat{h}_+(\Homv))=\mu(\Homv)$.
\end{proof}

\begin{lemma} \label{Lem:NuConstant}
	Suppose for $i=1,2,3$ that $\gamma_{m_i,n_i}\in\calB$ with
	\[
		\mu(\gamma_{m_1,n_1})=\mu(\gamma_{m_2,n_2})=\mu(\gamma_{m_3,n_3}) .
	\]
	If the points $\{ (m_i, n_i) \mid i = 1,2,3 \}$ are not collinear in $\ZZ^2$ then $\mu$ is constant on $\calB$.
\end{lemma}

\begin{proof}
	By \reflem{NuExponential}, there are numbers $a,b\in\RR^+$ such that
	\[
		\mu(\gamma_{m,n})=\mu(\gamma_{0,0})a^mb^n
	\] 
	for all integer pairs $m$ and $n$.
	Then 
	\[
		\log\circ\mu(\gamma_{m,n})=\log \circ \mu(\gamma_{0,0}) + m \log(a) + n \log(b)
	\] 
	Thus $\log\circ\mu$ is a linear function of $m$ and $n$. Therefore if it is constant at three non-collinear points then it is constant.
\end{proof}

\begin{proof}[Proof of \refthm{G3Convex}]
	Suppose that $M$ hyperbolic so $h$ is a pseudo-Anosov map.
	Consider the action $\hat{h}_+$ on $\calB$.
	The induced map $h_+$ doesn't preserve any proper non-trivial subgroups of $H_1(T;\ZZ)=\ZZ^2$.
	Thus, none of the orbits of $\calB$ under $\hat{h}_+$ contain all collinear points unless the orbit is a single fixed point.

	Consider $\gamma_{m,n}\in\calB-\{\gamma_{0,0}\}$.
	Since $\gamma_{m,n}$ is not a fixed point of $\hat{h}_+$, the orbit of $\gamma_{m,n}$ contains three distinct elements $\gamma$, $\gamma'$, and $\gamma''$ which are not collinear.
	By \reflem{NuOrbit},
	\[
		\mu(\gamma)=\mu(\gamma')=\mu(\gamma'') .
	\]
	Therefore, by \reflem{NuConstant} the function $\mu$ is constant on $\calB$,
	and by \reflem{ConstantG3} we have that $C=G_3$.
\end{proof}

\section{A non-standard bi-order}
\label{Sec:nonstandard}

To prove \refthm{nonstandard}, we construct non-standard bi-orders on untwisted hyperbolic punctured torus bundles following \refcon{BiorderConstructionSchemeFin}.

\subsection{The first quotient}
\label{Sec:FirstQuotient}

The map $h \from T \to T$ induces a map $h_+\in\SL(2,\ZZ)$ on $H_1(T; \ZZ) \cong \ZZ^2$  with eigenvalues $(\lambda, 1/\lambda)$ and corresponding eigenbasis $(e_\lambda, e_{1/\lambda})$.
The bundle is untwisted, so the eigenvalues are positive. 
Using this eigenbasis, we can write $h_+$ as the matrix 
\[
    \left(\begin{array}{cc}
        \lambda & 0 \\
        0 & 1/\lambda
    \end{array}\right) .
\]
It follows that $h_+$ preserves four positive cones $Q_1$ of $\ZZ^2$.
Each of these is given by taking one of $\{e_\lambda, e_{1/\lambda}, -e_\lambda, -e_{1/\lambda}\}$ for the vector $v$ in \refexa{OrderZxZ}\refitm{IrrationalSlope}.
See \reffig{PositiveConesForZxZ}.

\begin{figure}[htbp]
\subfloat[$e_\lambda$]
{
\includegraphics[width = 0.4\textwidth]{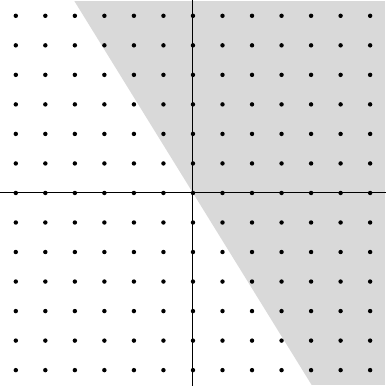}
}
\quad
\subfloat[$-e_\lambda$]{
\includegraphics[width = 0.4\textwidth]{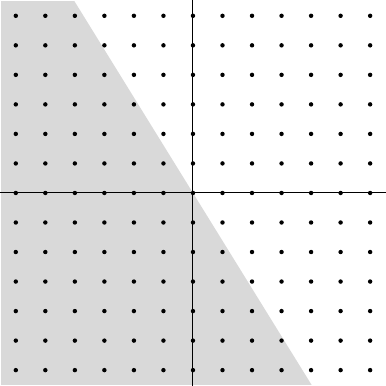}
}

\subfloat[$e_{1/\lambda}$]
{
\includegraphics[width = 0.4\textwidth]{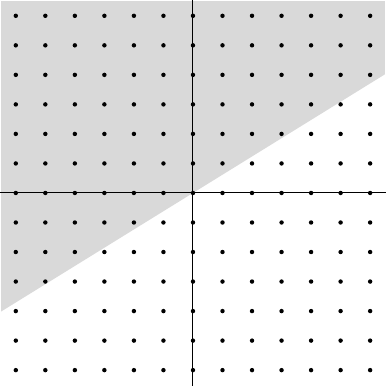}
}
\quad
\subfloat[$-e_{1/\lambda}$]{
\includegraphics[width = 0.4\textwidth]{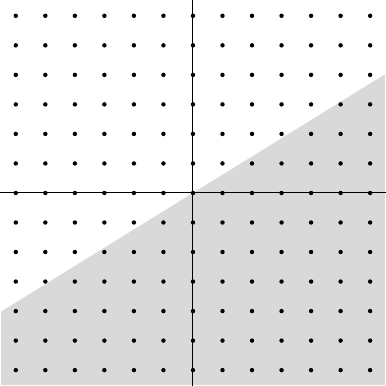}
}
\caption{The four positive cones (shaded) in the example of the figure eight knot complement.}
\label{Fig:PositiveConesForZxZ}
\end{figure}

\subsection{The second quotient}
\label{Sec:SecondQuotient}

Recall the definition of the map $\hat{h}$ and basis of $\calB$ from Definitions~\ref{Def:hHat} and~\ref{Def:Basis}.
The goal is to define a positive cone of $H_1(\hat{T};\ZZ)$ invariant under  $\hat{h}_+$.
In short, the plan is to find an ordering on  $\calB$ that is invariant under $\hat{h}_+$,
and then define a bi-ordering of $H_1(\hat{T};\ZZ)$ in a lexicographic fashion relative to the ordering of $\calB$ as in \refexa{InfiniteLex}.

\begin{definition}
Let $e$ be one of $\{e_\lambda, e_{1/\lambda}, -e_\lambda, -e_{1/\lambda}\}$.
Again using \refexa{OrderZxZ}\refitm{IrrationalSlope}, we generate an order $\prec_e$ on $\ZZ^2$ (possibly different from the choice made in \refsec{FirstQuotient}). 
This order induces an ordering $<_e$ on $\mathcal{B}$ via $q$.
That is, 
$\gamma_1 <_e \gamma_2$ if and only if $q(\gamma_1) \prec_e q(\gamma_2)$.
\end{definition}

\begin{lemma}
\label{Lem:ShiftInvariant}
The order $<_e$ on $\calB$ is invariant under shifting by any vector $(m,n) \in \ZZ^2$. 
\end{lemma}

\begin{proof}
Addition preserves the order $\prec_e$ on $\ZZ^2$.
The result then follows from \refdef{Shift}.
\end{proof}

\begin{lemma}	
\label{Lem:BOrder}
The order $<_e$ on $\calB$ is invariant under $\hat{h}_+$. 
\end{lemma}

\begin{proof}
Let $\gamma_1$ and $\gamma_2$ be elements of $\calB$.
Suppose that $\gamma_1 <_e \gamma_2$.
Therefore $q(\gamma_1) \prec_e q(\gamma_2)$.
Since $\prec_e$ is invariant under the action of $h_+$ we have that 
\[
h_+\circ q(\gamma_1) \prec_e h_+\circ q(\gamma_2).
\]
By \reflem{ActionOfHatH+}\refitm{Commutes}, we have 
\[
q \circ \hat{h}_+(\gamma_1) \prec_e q \circ \hat{h}_+(\gamma_2),
\]
and so
\[
\hat{h}_+(\gamma_1) <_e \hat{h}_+(\gamma_2). \qedhere
\]
\end{proof}

We next define a positive cone $Q_2$ on $H_1(\hat{T};\ZZ)$ as follows.
\begin{definition}
\label{Def:Q2}
Given a non-zero element $\Homv\in H_1(\hat{T};\ZZ)$, we define $\max(\Homv)$ to be the maximum (under $<_e$) basis element $\gamma\in\mathcal{B}$ such that the coefficient of $\gamma$ in $\Homv$ is nonzero.
Let $k(\Homv)$ be this coefficient. 
We then define
\[
Q_2 = \{ \Homv \in H_1(\hat{T};\ZZ) \mid k(\Homv) > 0\}.  \qedhere
\]
\end{definition}

\begin{lemma}
	\label{Lem:OrderInvariant}
	The set $Q_2$ is a positive cone. 
	Moreover $Q_2$ is invariant under the action of $\hat{h}_+$.
\end{lemma}

\begin{proof}
	We check the three properties in \refdef{PositiveCone}.
	First suppose that $\Homv, \Homw \in Q_2$. 
	If $\max(\Homv) = \max(\Homw)$ then $\max(\Homv+\Homw) = \max(\Homv) =  \max(\Homw)$ and $k(\Homv+\Homw) = k(\Homv) + k(\Homw) > 0$. If not, then breaking symmetry, assume that $\max(\Homv) <_e \max(\Homw)$. In this case, $\max(\Homv+\Homw) = \max(\Homw)$ and $k(\Homv+\Homw) = k(\Homw) > 0$.
	This proves \refitm{Closed}.
	Now suppose that $\Homv \in H_1(\hat{T};\ZZ) - \{0\}$. 
	Note that $\max(-\Homv) = \max(\Homv)$ and $k(-\Homv) = -k(\Homv)$.
	Therefore $\Homv \in Q_2$ if and only if $-\Homv \in Q_2^{-1}$.
	By definition, $0 \notin Q_2 \cup Q_2^{-1}$, so we have \refitm{Partition}.
	Since $H_1(\hat{T};\ZZ)$ is abelian, we have \refitm{Invariant}.
	
	Finally we show that $Q_2$ is invariant under  $\hat{h}_+$.
	Let $\Homv\in Q_2$. Thus $k(\Homv) > 0$.
	For some set of coefficients $\{k_\gamma\}_{\gamma\in\mathcal{B}}$,
	\[
		\Homv=\sum_{\gamma\in\calB} k_\gamma \gamma.
	\]
	Thus,
	\[
		\hat{h}_+(\Homv)=\sum_{\gamma\in\calB} k_\gamma \hat{h}_+(\gamma) .
	\]

	By \reflem{BOrder}, the ordering of elements in $\mathcal{B}$ is invariant under $\hat{h}_+$.
	Therefore, $\max(\Homv)=\max(\hat{h}_+(\Homv))$.
	Thus $k(\hat{h}_+(\Homv)) = k(\Homv) > 0$. 
	Therefore $\hat{h}_+(\Homv)\in Q_2$.
\end{proof}

By the definition of $\hat{T}$, the fundamental group $\pi_1(\hat{T})$ is isomorphic to $G_2$. 
Abelianizing, we see that  $G_2/[G_2,G_2]$ is naturally isomorphic to $H_1(\hat{T},\ZZ)$.
Under this isomorphism, $Q_2$ defines a positive cone of $G_2/[G_2,G_2]$.

\subsection{The remainder}

We are now ready to build a bi-ordering of the free group $G$ that is invariant under the action of $h$.
We require one more piece of data: we choose an arbitrary positive cone of $\pi_1(M)$ (recall that in the statement of \refthm{nonstandard} we assume that $\pi_1(M)$ is bi-orderable). 
We refer to the restriction of this arbitrary positive cone to $[G_2,G_2]$ as $Q_3$.

\subsection{Putting it all together}

We now apply \refcon{BiorderConstructionSchemeFin} to build a bi-ordering of the free group $G$, invariant under the action of $h$, as follows.
Recall that $p_1:G\to\ZZ^2$ is the abelianization map and that $p_2:G_2\to G_2/[G_2,G_2]$ is the quotient map.

\begin{definition}
\label{Def:ConstructedPositiveCone}
Define $P$ to be the set of all $g\in G$ such that either
\begin{enumerate}
    \item $p_1(g)\in Q_1$, or
    \item $g\in G_2$ and $p_2(g)\in Q_2$, or
    \item $g\in Q_3$. \qedhere
\end{enumerate}
\end{definition}

\begin{lemma}
\label{Lem:PPositiveCone}
    $P$ is a positive cone of $G$.
\end{lemma}

\begin{proof}
	The goal is to show that $P$ is a positive cone by repeatedly applying \reflem{constructpositivecone}.
	By definition, $Q_3$ is a positive cone of $[G_2, G_2]$ and is invariant under conjugation.
	We first apply \reflem{constructpositivecone} with $p_2$ as the surjective homomorphism, mapping from $G_2$ to $G_2/[G_2, G_2]$.  
	We obtain that $P \cap G_2$ is a positive cone on $G_2$.
	
	Next, we want to apply \reflem{constructpositivecone} with $p_1$ as the surjective homomorphism, mapping from $G$ to $G/G_2 \cong H_1(T;\ZZ) \cong \ZZ^2$.  
	In order to do so, we need to show that $P \cap G_2$ is invariant under conjugation by the elements of $G$.
	
	Suppose that $g \in G_2 \cap P$ and that $x \in G$.
	There are two cases. 
	If $g \in Q_3$ then again by definition $x g x^{-1}  \in Q_3 \subset G_2 \cap P$.
	Otherwise we have that $g \in G_2$ and $p_2(g) \in Q_2$.
	In this case, suppose that the homology class $p_1(x)$ is $m[\alpha]+n[\beta]\in H_1(T,\ZZ)$,
	and the homology class $p_2(g)$ in $H_1(\hat{T},\ZZ)$ is
	\[
		\sum_{\gamma\in\calB} k_\gamma \gamma.  
	\]
	By \reflem{Shift}, the homology class of $p_2(x g x^{-1})$ in $H_1(\hat{T},\ZZ)$ is 
	\[
		\sum_{\gamma\in\calB} k_\gamma s_{(m,n)}(\gamma).
	\]
		
	\reflem{ShiftInvariant} implies that the maximal element of $\calB$ appearing in the sum for $x^{-1}gx$ is the shift of the maximal element appearing in the sum for $g$. The coefficients for these elements are equal, and therefore have the same sign.
	By \refdef{Q2}, we have that $p_2(g)\in Q_2$ if and only if $p_2(x^{-1}gx)\in Q_2$.
	Thus we have obtained the hypotheses of \reflem{constructpositivecone} and we are done.
\end{proof}

\begin{lemma}
    The positive cone $P$ is invariant under the map induced by the monodromy $h_*$.
\end{lemma}

\begin{proof}
	Suppose that $g\in P$.
	There are three possibilities, as given in \refdef{ConstructedPositiveCone}.
	First suppose that $p_1(g)\in Q_1$.
	Since $p_1 \circ h_*(g)=h_+ \circ p_1(g)$ and by definition, $h_+$ preserves $Q_1$, we have that $p_1(h_*(g))\in Q_1$.
	
	Next suppose that $g\in G_2$ and $p_2(g)\in Q_2$.
	Since $G_2$ is characteristic, $h_*(g)\in G_2$.
	The positive cone $Q_2$ is invariant under $\hat{h}_*$ by \reflem{OrderInvariant}, so $\hat{h}_* \circ p_2(g) \in Q_2$.
	Therefore $p_2 \circ h_*(g)=\hat{h}_* \circ p_2(g) \in Q_2$.

	Finally, suppose that $g\in Q_3$.
	Since $Q_3$ is the restriction of a positive cone of $\pi_1(M)$ and $[G_2,G_2]$ is characteristic, $Q_3$ is invariant under $h_*$.
	Thus, $h_*(g)\in Q_3$.
\end{proof}

\begin{lemma}
\label{Lem:G3NotConvex}
	$G_3 = [G_2, G]$  in the lower central series for $G$ is not convex with respect to $P$.
\end{lemma}

\begin{proof}
We will apply \refprop{CosetOrder}, finding elements $y \in P$ and $z \in P^{-1}$ in the same nontrivial coset of $G_3$.
	
	Let $\gamma_0=\alpha\beta\alpha^{-1}\beta^{-1}$, and let $\gamma_1=\beta\alpha^{-1}\beta^{-1}\alpha \in G_2 = [G,G]$.
	Since $\gamma_0=[\alpha,\gamma_1]\gamma_1$, we have that $\gamma_0$ and $\gamma_1$ are in the same coset $S = \gamma_0 G_3 = \gamma_1 G_3$. (Recall that $G_3$ is normal, so left and right cosets are the same.)
	Consider the elements $y=\gamma_0^2\gamma_1^{-1}$ and $z=\gamma_0^{-1}\gamma_1^2$.
	Since $\gamma_0$ and $\gamma_1$ are in $S$, we have that $\gamma_0 \gamma_1^{-1}$ and $\gamma_0^{-1}\gamma_1$ are in $G_3$. Thus $y$ and $z$ are also in $S$.
	
	By the definition of $\hat{T}$, we have that $G_2\cong\pi_1(\hat{T},\hat{x})$.
	We define the \emph{winding number} of an element $f \in G_2$ by
	\[
		w(f)=\sum_{\gamma \in \calB}k_\gamma
	\]
	where
	\[
		p_2(f) = \sum_{\gamma\in\calB} k_\gamma \gamma.
	\]
	The group $G_3$ is generated by terms $[g,f]$ where $g\in G$ and $f\in G_2$.
	By \reflem{Shift}, we have that $w(g f g^{-1}) = w(f)$, so $w([g,f])=0$.
	Therefore $G_3$ is contained in $\ker(w)$.
	Since $w(\gamma_0)= w(\gamma_1)=1$, we have that $S \neq G_3$.
	
	The homology class $p_2(y)=2 p_2(\gamma_0) - p_2(\gamma_1)$, and the homology class $p_2(z) = -p_2(\gamma_0) + 2 p_2(\gamma_1)$.
	Since $p_2(\gamma_0)$ and $p_2(\gamma_1)$ are in $\mathcal{B}$ and the signs of the coefficients of $p_2(y)$ and $p_2(z)$ are opposite,
	exactly one of the homology classes is in $Q_2$ and the other is in $Q_2^{-1}$.
	Therefore, one of $y$ and $z$ is in $P$ and the other is in $P^{-1}$.
\end{proof}

\begin{proof}[Proof of \refthm{nonstandard}]
\reflem{PPositiveCone} tells us that $P$ is a positive cone. 
From \reflem{G3NotConvex} and \refprop{StandardConvexDef} we deduce that $P$ is not the positive cone of a standard bi-order of $G$.
\end{proof}

\begin{corollary}
\label{Cor:NoMaximalConvex}
Let $P$ be the positive cone given in \refdef{ConstructedPositiveCone}.
Then $G_2$ has no maximal convex subgroup with respect to $P$.
\end{corollary}

\begin{proof}
This follows from \reflem{G3NotConvex} and \refthm{G3Convex}.
It can also be seen directly by exhibiting an increasing sequence of convex subgroups. Namely,
for each $\gamma \in \calB$, let 
\[
	C_\gamma = \{ g \in G_2 \mid \max(p_2(g)) = \gamma \} \cup \{ \text{id}_G \}
\]
where $m$ is the function given in \refdef{Q2}.
Each $C_\gamma$ is a convex subgroup, $\gamma <_e \delta$ implies that $C_\gamma \subsetneq C_\delta \subsetneq G_2$, and $\bigcup_{\gamma \in \calB} C_\gamma = G_2$.
\end{proof}

\bibliographystyle{hyperplain}
\bibliography{non-standard_biorders}

\begin{thebibliography}{10}

\bibitem{BGW13}
Steven Boyer, Cameron~McA. Gordon, and Liam Watson.
\newblock On {L}-spaces and left-orderable fundamental groups.
\newblock {\em Mathematische Annalen}, 356(4):1213--1245, 2013.

\bibitem{BRW05}
Steven Boyer, Dale Rolfsen, and Bert Wiest.
\newblock Orderable 3-manifold groups.
\newblock {\em Annales de l'Institut Fourier}, 55(1):243--288, 2005.

\bibitem{CDN16}
Adam Clay, Colin Desmarais, and Patrick Naylor.
\newblock Testing bi-orderability of knot groups.
\newblock {\em Canadian Mathematical Bulletin}, 59(3):472–482, 2016.

\bibitem{ClaRol12}
Adam Clay and Dale Rolfsen.
\newblock Ordered groups, eigenvalues, knots, surgery and {L}-spaces.
\newblock {\em Mathematical Proceedings of the Cambridge Philosophical
  Society}, 152(1):115–129, 2012.

\bibitem{ClayRolf16}
Adam Clay and Dale Rolfsen.
\newblock {\em Ordered Groups and Topology}.
\newblock American Mathematical Society, Providence, Rhode Island, 2016.

\bibitem{ClementMajewicsZyman17}
Anthony~E. Clement, Stephen Majewicz, and Marcos Zyman.
\newblock {\em The Collection Process and Basic Commutators}, pages 75--106.
\newblock Springer International Publishing, Cham, 2017.

\bibitem{Con59}
Paul Conrad.
\newblock {Right-ordered groups.}
\newblock {\em Michigan Mathematical Journal}, 6(3):267 -- 275, 1959.

\bibitem{FarMar12}
Benson Farb and Dan Margalit.
\newblock {\em A Primer on Mapping Class Groups (PMS-49)}.
\newblock Princeton University Press, 2012.

\bibitem{Hall34}
P.~Hall.
\newblock A contribution to the theory of groups of prime-power order.
\newblock {\em Proceedings of the London Mathematical Society},
  s2-36(1):29--95, 1934.

\bibitem{Holder1901}
O.~H{\"o}lder.
\newblock {\em Die Axiome der Quantit{\"a}t und die Lehre vom Mass}.
\newblock Teubner, 1901.

\bibitem{John23a}
Jonathan Johnson.
\newblock Residual torsion-free nilpotence, bi-orderability, and two-bridge
  links.
\newblock {\em Canadian Journal of Mathematics}, page 1–64, 2023.

\bibitem{John23b}
Jonathan Johnson.
\newblock Residual torsion-free nilpotence, biorderability and pretzel knots.
\newblock {\em Algebr. Geom. Topol.}, 23(4):1787--1830, 2023.

\bibitem{KinRolf18}
Eiko Kin and Dale Rolfsen.
\newblock Braids, orderings, and minimal volume cusped hyperbolic 3-manifolds.
\newblock {\em Groups Geom. Dyn.}, 12(3):961--1004, 2018.

\bibitem{Levi42}
F.~W. Levi.
\newblock Ordered groups.
\newblock {\em Proceedings of the Indian Academy of Sciences - Section A},
  16(4):256, Oct 1942.

\bibitem{LRR08}
Peter~A. Linnell, Akbar~H. Rhemtulla, and Dale~P.O. Rolfsen.
\newblock Invariant group orderings and {G}alois conjugates.
\newblock {\em Journal of Algebra}, 319(12):4891 -- 4898, 2008.

\bibitem{Magnus35}
Wilhelm Magnus.
\newblock Beziehungen zwischen gruppen und idealen in einem speziellen ring.
\newblock {\em Mathematische Annalen}, 111(1):259--280, 1935.

\bibitem{PerRolf03}
Bernard Perron and Dale Rolfsen.
\newblock On orderability of fibred knot groups.
\newblock {\em Mathematical Proceedings of the Cambridge Philosophical
  Society}, 135(1):147–153, 2003.

\bibitem{Yam17}
Takafumi Yamada.
\newblock A family of bi-orderable non-fibered 2-bridge knot groups.
\newblock {\em Journal of Knot Theory and Its Ramifications}, 26(04):1750011,
  2017.

\end{thebibliography}

\end{document}